\documentclass[a4paper]{scrartcl}

\usepackage[utf8]{inputenc}
\usepackage[T1]{fontenc}

\usepackage{amsmath, amsthm, amssymb}
\usepackage{mathtools}
\usepackage{url}

\usepackage{enumitem}
\newlist{hypothenum}{enumerate}{3}
\setlist[hypothenum,1]{label=(\roman*)}

\theoremstyle{plain}

\newtheorem{theorem}{Theorem}[section]

\newtheorem{problem}[theorem]{Problem}
\newtheorem{conjecture}[theorem]{Conjecture}
\theoremstyle{definition}

\theoremstyle{remark}

\numberwithin{equation}{section}

\newcommand{\setsuch}[2]{\left\{ #1 \; \middle| \; #2 \right\}}
\newcommand{\restr}[2]{{\left. #1 \right|}_{#2}}
\newcommand{\RR}{\mathbb{R}}
\newcommand{\CC}{\mathbb{C}}

\newcommand{\lie}{\mathfrak}

\DeclareMathOperator{\Id}{Id}

\DeclareMathOperator{\SO}{SO}

\usepackage[all,cmtip]{xy}
\newcommand{\fundef}[5]{
\entrymodifiers={+!!<0pt,\fontdimen22\textfont2>}
\xymatrix@R=3pt{\llap{$#1$\;\;} {#2} \ar@{->}[r] & {#3} \\ {#4} \ar@{|->}[r] & {#5}}
} %Formats function definitions nicely. For more complicated things (bijections with two arrows etc.), redo it by hand. To explain the "\entrymodifiers" line: see https://www.math.lsu.edu/~aperlis/publications/axisalignment/perlis_axisalignment_24Jun2003.pdf
\def\noqed{\renewcommand{\qedsymbol}{}}
\newcommand{\ie}{i.e.\ }
\newcommand{\eg}{e.g.\ }

\begin{document}

\title{Action of the restricted Weyl group on the $L$-invariant vectors of a representation}
\author{Ilia Smilga}
              
\maketitle

\begin{abstract}
This note constitutes a brief survey of our recent work on the problem of determining, for a given real Lie group~$G$, the set of representations~$V$ in which the longest element~$w_0$ of the restricted Weyl group~$W$ acts nontrivially on the subspace~$V^L$ of~$V$ formed by vectors that are invariant by~$L$, the centralizer of a maximal split torus of~$G$.
\end{abstract}

\section{Introduction and motivation}

\subsection{Basic notations and statement of problem}

Let $G$ be a semisimple real Lie group, $\lie{g}$ its Lie algebra, $\lie{g}^\CC$ the complexification of~$\lie{g}$. We start by establishing the notations for some well-known objects related to~$\lie{g}$.

\begin{itemize}
\item We choose in~$\lie{g}$ a \emph{Cartan subspace}~$\lie{a}$ (an abelian subalgebra of $\lie{g}$ whose elements are diagonalizable over~$\RR$ and which is maximal for these properties).
\item We choose in~$\lie{g}^\CC$ a \emph{Cartan subalgebra}~$\lie{h}^\CC$ (an abelian subalgebra of $\lie{g}^\CC$ whose elements are diagonalizable and which is maximal for these properties) that contains~$\lie{a}$.
\item We denote $L := Z_G(\lie{a})$ the centralizer of~$\lie{a}$ in~$G$, $\lie{l}$ its Lie algebra.
\item Let $\Delta$~be the set of roots of~$\lie{g}^\CC$ in~$(\lie{h}^\CC)^*$. We shall identify $(\lie{h}^\CC)^*$ with~$\lie{h}^\CC$ via the Killing form. We call $\lie{h}_{(\RR)}$ the $\RR$-linear span of~$\Delta$; it is given by the formula $\lie{h}_{(\RR)} = \lie{a} \oplus i \lie{a}^\perp$.
\item We choose on~$\lie{h}_{(\RR)}$ a lexicographical ordering that ``puts $\lie{a}$ first'', \ie such that every vector whose orthogonal projection onto~$\lie{a}$ is positive is itself positive. We call $\Delta^+$ the set of roots in~$\Delta$ that are positive with respect to this ordering, and we let $\Pi = \{\alpha_1, \ldots, \alpha_r\}$ be the set of simple roots in~$\Delta^+$. Let~$\varpi_1, \ldots, \varpi_r$ be the corresponding fundamental weights.
\item We introduce the \emph{dominant Weyl chamber}
\[\lie{h}^+ := \setsuch{X \in \lie{h}}{\forall i = 1, \ldots, r,\quad \alpha_i(X) \geq 0},\]
and the \emph{dominant restricted Weyl chamber}
\[\lie{a}^+ := \lie{h}^+ \cap \lie{a}.\]
\item We introduce the \emph{restricted Weyl group} $W := N_G(\lie{a})/Z_G(\lie{a})$ of~$G$. Then $\lie{a}^+$ is a fundamental domain for the action of~$W$ on~$\lie{a}$. We define the \emph{longest element} of the restricted Weyl group as the unique element $w_0 \in W$ such that $w_0(\lie{a}^+) = -\lie{a}^+$.
\end{itemize}

Our goal is to study the action of~$W$, and more specifically of~$w_0$, on various representations~$V$ of~$G$. Note however that this action is ill-defined: indeed if we want to see the abstract element $w_0 \in W = N_G(\lie{a})/Z_G(\lie{a})$ as the projection of some concrete map $\tilde{w_0} \in N_G(\lie{a}) \in G$, then $\tilde{w}_0$ is defined only up to multiplication by an element of $Z_G(\lie{a}) = L$, whose action on~$V$ can of course be nontrivial.

This naturally suggests the idea of restricting to $L$-invariant vectors. Given a representation $V$ of~$\lie{g}$, we denote
\[V^L := \setsuch{v \in V}{\forall l \in L,\;\; l \cdot v = v}\]
the $L$-invariant subspace of~$V$: then $W$, and in particular $w_0$, has a well-defined action on~$V^L$.

Our goal is to solve the following problem:
\begin{problem}
\label{nontrivial_everything_group}
Given a semisimple real Lie group~$G$, characterize the representations $V$ of~$G$ for which the action of~$w_0$ on $V^L$ is nontrivial.
\end{problem}

In this note, we shall present our recent work on this problem.

\subsection{Background and motivation}

This problem arose from the author's work in geometry. The interest of this particular algebraic property is that it furnishes a sufficient, and presumably necessary, condition for another, geometric property of~$V$. Namely, the author obtained the following result:
\begin{theorem}{\cite{Smi16b}}
Let $G$~be a semisimple real Lie group, $V$ a representation of~$G$. Suppose that the action of~$w_0$ on~$V^L$ is nontrivial. Then there exists, in the affine group $G \ltimes V$, a subgroup $\Gamma$ whose linear part is Zariski-dense in~$G$, which is free of rank at least~$2$, and acts properly discontinuously on the affine space corresponding to~$V$.
\end{theorem}
He, and other people, also proved the converse statement in some special cases:
\begin{theorem}
The converse holds, for irreducible $V$:
\begin{itemize}
\item \cite{Smi18} if $G$~is split, but not of type $A_n$ ($n \geq 2$), $D_{2n+1}$ or~$E_6$;
\item \cite{Smi18} if $G$~is split, has one of these types, and $V$ satisfies a very restrictive additional assumption (see \cite{Smi18} for the precise statement);
\item \cite{AMS11} if $G = \SO(p,q)$ for arbitrary $p$ and~$q$, and $V = \RR^{p+q}$ is the standard representation.
\end{itemize}
\end{theorem}
Moreover, it seems plausible that, by combining the approaches of~\cite{Smi18} and~\cite{AMS11}, we might prove the converse in all generality. This geometric property is related to the so-called Auslander conjecture \cite{Aus64}, which is an important conjecture that has stood for more than fifty years and generated an enormous amount of work: see \eg \cite{Mil77, Mar83, FG83, AMS02, DGKPre} and many many others. For the statement of the conjecture as well as a more comprehensive survey of past work on it, we refer to~\cite{AbSur}.

\subsection{Acknowledgements}

The author is supported by the European Research Council (ERC) under the European Union Horizon 2020 research and innovation programme, grant
647133 (ICHAOS).

\section{Basic properties}

\subsection{Reduction from groups to algebras}

First of all, note that, without loss of generality, we may restrict our attention to irreducible representations: indeed, plainly, the condition $\restr{w_0}{V^L} \neq \Id$ holds for a direct sum $V = V_1 \oplus \cdots \oplus V_k$ if and only if it holds for at least one of the direct summands~$V_i$.

Let us start by recalling the classification of the irreducible representations of $G$, $\lie{g}$ and~$\lie{g}^\CC$. We introduce the notation $P$ (resp. $Q$) for the \emph{weight lattice} (resp. \emph{root lattice}), \ie the abelian subgroup of~$\lie{h}_{(\RR)}$ (or, technically, its dual) generated by $\varpi_1, \ldots, \varpi_r$ (resp. by~$\Delta$).
\begin{theorem}[see \eg {\cite[Theorem~5.5]{Kna96} or \cite[Theorems 9.4 and~9.5]{Hall15}}]
To every irreducible representation of~$\lie{g}^\CC$ we may associate, in a bijective way, a vector $\lambda \in P \cap \lie{h}^+$ called its \emph{highest weight}.
\end{theorem}
\begin{theorem}[see \eg {\cite[Proposition~7.15]{Kna96}}]
The restriction map $\rho \mapsto \restr{\rho}{\lie{g}}$ induces a bijection between irreducible representations of~$\lie{g}^\CC$ and those of~$\lie{g}$.
\end{theorem}
\begin{theorem}
Every irreducible representation of~$G$ yields, by derivation, an irreducible~\cite[Proposition~7.15]{Kna96} representation of~$\lie{g}$. The converse is true if $G = \tilde{G}$ is simply-connected. For arbitrary $G$, the representation $\rho_\lambda(\lie{g})$ lifts to~$G$ if and only if \cite[\S~5.8]{Kna96} $\lambda$ lies in some lattice $\Lambda_G$ satisfying $Q \subset \Lambda_G \subset P$.
\end{theorem}
For every dominant weight $\lambda \in P \cap \lie{h}^+$, we denote by $V_\lambda$ the representation of~$\lie{g}^\CC$, of~$\lie{g}$ or (if it exists) of~$G$ with highest weight~$\lambda$.

To reformulate Problem~\ref{nontrivial_everything_group} in terms of algebras, it remains to note two things. First, we note that the action of $w_0$ on~$V^\lie{l}$ depends only on~$\lie{g}$, not on~$G$: indeed $G$ is in general the quotient of the (unique) simply-connected group $\tilde{G}$ with algebra~$\lie{g}$ by some subgroup of its center; but the center of~$\tilde{G}$ is in particular contained in the centralizer~$L$ of~$\lie{a}$, so is irrelevant when acting on~$V^L$. Second, it is easy to see that the space~$V^L$ always coincides with the space
\[V^\lie{l} := \setsuch{v \in V}{\forall X \in \lie{l},\;\; X \cdot v = 0}\]
of $\lie{l}$-invariant vectors of~$V$. So Problem~\ref{nontrivial_everything_group} is in fact equivalent to the following:

\begin{problem}
\label{nontrivial_everything}
Given a semisimple real Lie algebra~$\lie{g}$, characterize the set of weights $\lambda \in P \cap \lie{h}^+$ for which the action of~$w_0$ on the $\lie{l}$-invariant subspace~$V_\lambda^\lie{l}$ of the representation~$V_\lambda$ of~$\lie{g}$ with highest weight~$\lambda$ is nontrivial.
\end{problem}

\subsection{Additivity properties}

A first step towards the solution of Problem~\ref{nontrivial_everything} is given by the following characterization:
\begin{theorem}~
\label{submonoid_and_ideal}
\begin{hypothenum}
\item \label{itm:radical} The set of weights~$\lambda \in P \cap \lie{h}^+$ such that $V_\lambda^\lie{l} \neq 0$ is contained in $Q \cap \lie{h}^+$.
\item \label{itm:submonoid} The set of weights~$\lambda$ such that $V_\lambda^\lie{l} \neq 0$ is in fact a submonoid of the additive monoid $Q \cap \lie{h}^+$, \ie
\[\forall \lambda, \mu \in Q \cap \lie{h}^+,\quad
\begin{cases}
  V_\lambda^\lie{l} \neq 0 \\
  V_\mu^\lie{l} \neq 0
\end{cases}
  \implies
V_{\lambda+\mu}^\lie{l} \neq 0.\]
\item \label{itm:ideal} In this monoid, the subset of weights~$\lambda$ such that $\restr{w_0}{V_\lambda^\lie{l}} \neq \pm \Id$ is an ideal, \ie
\[\forall \lambda, \mu \in Q \cap \lie{h}^+,\quad
\begin{cases}
  \restr{w_0}{V_\lambda^\lie{l}} \neq \pm \Id \\
  V_\mu^\lie{l} \neq 0
\end{cases}
\implies
\restr{w_0}{V_{\lambda+\mu}^\lie{l}} \neq \pm \Id.\]
\end{hypothenum}
\end{theorem}
\begin{proof}
Point (i) is straightforward: indeed, since $\lie{h}$ is an abelian algebra containing~$\lie{a}$, we have $\lie{h} \subset \lie{l}$, hence $V^\lie{l}_\lambda$ is contained in~$V^\lie{h}_\lambda$, which is just the zero-weight space of~$V_\lambda$. By well-known theory (see \eg \cite{Hall15}, Theorem~10.1), the latter is nontrivial, or in other terms $0$~is a weight of~$V_\lambda$, if and only if $\lambda$ lies in the root lattice~$Q$.

Points (ii) and~(iii) rely on the following classical theorem. \noqed
\end{proof}

Let $G$ be a simply-connected complex Lie group with Lie algebra $\lie{g}^\CC$ and $N$ a maximal unipotent subgroup of~$G$.
Define $\CC[G/N]$ the space of regular (\ie polynomial) functions on $G/N$.
Pointwise multiplication of functions is $G$-equivariant and makes $\CC[G/N]$ into a $\CC$-algebra without zero divisors (because $G/N$ is irreducible as an algebraic variety).
\begin{theorem}[see \eg {\cite[(3.20)--(3.21)]{VinPopBook}}]\label{thm:Vinberg}
  Each finite-dimensional representation of~$G$ (or equivalently of its Lie algebra~$\lie{g}^\CC$) occurs exactly once as a direct summand of the representation $\CC[G/N]$.
  The $\CC$-algebra $\CC[G/N]$ is graded by the highest weight $\lambda$, in the sense that the product of a vector in $V_\lambda$ by a vector in $V_\mu$ lies in $V_{\lambda+\mu}$ (where $V_\lambda$ stands here for the subrepresentation of $\CC[G/N]$ with highest weight $\lambda$).
\end{theorem}
For given $\lambda$ and $\mu$, we call \emph{Cartan product} the induced bilinear map $\odot: V_\lambda \times V_\mu\to V_{\lambda+\mu}$.
Given $u\in V_\lambda$ and $v\in V_\mu$, this defines $u\odot v\in V_{\lambda+\mu}$ as the projection of $u\otimes v\in V_\lambda\otimes V_\mu=V_{\lambda+\mu}\oplus\dots$.
Since $\CC[G/N]$ has no zero divisor, $u\odot v\neq 0$ whenever $u\neq 0$ and $v\neq 0$.

We may now finish the proof:

\begin{proof}[Proof of Theorem~\ref{submonoid_and_ideal}, continued]~
\begin{hypothenum}[start=2]
\item Let $\lambda$ and~$\mu$ be two weights with this property. Choose any two nonzero vectors $u_1$ and~$u_2$ in $V_{\lambda_1}^{\lie{l}}$ and~$V_{\lambda_2}^{\lie{l}}$ respectively. Then the vector $u_1 \odot u_2$ is in $V_{\lambda_1 + \lambda_2}$, is invariant by~$\lie{l}$, and is still nonzero.
\item Let $\lambda$ be such that $\restr{w_0}{V_\lambda^\lie{l}} \neq \pm \Id$, and $\mu$ be such that $V_\mu^\lie{l} \neq 0$. We can then choose nonzero vectors $u_+$ and $u_-$ in~$V_\lambda^\lie{l}$ such that $w_0\cdot u_+=u_+$ and $w_0\cdot u_-=-u_-$, and a nonzero vector $v \in V_\mu^\lie{l}$ such that $w_0\cdot v=\pm v$ for some sign (indeed since $w_0^2 = \Id$ in the Weyl group and the action of the Weyl group on~$V^\lie{l}$ is well-defined, $\restr{w_0}{V_\mu^\lie{l}}$ is a linear involution).
  Then $u_+\odot v$ and $u_-\odot v$ are nonzero elements of~$V_{\lambda+\mu}^\lie{l}$ on which $w_0$ acts by opposite signs. \qedhere
\end{hypothenum}
\end{proof}

\subsection{Reduction from semisimple to simple algebras}

Theorem~\ref{submonoid_and_ideal} suggests the decomposition of Problem~\ref{nontrivial_everything} into two subproblems, as follows:
\begin{problem}
\label{nontrivial_Vl}
Given a semisimple Lie algebra~$\lie{g}$ and a weight $\lambda \in P \cap \lie{h}^+$, give a simple necessary and sufficient condition for having $V_\lambda^\lie{l} \neq 0$.
\end{problem}
\begin{problem}
\label{nontrivial_w0_action}
Given a semisimple Lie algebra~$\lie{g}$ and a weight $\lambda \in P \cap \lie{h}^+$, assuming that $V_\lambda^\lie{l} \neq 0$, give a simple necessary and sufficient condition for having $\restr{w_0}{V_\lambda^\lie{l}} = \Id$.
\end{problem}

Let us now reduce these two problems to the case where $\lie{g}$ is simple. This can be done by using the following theorem, whose proof is straightforward:
\begin{theorem}
Let $\lie{g} = \lie{g}_1 \oplus \cdots \oplus \lie{g}_k$ is a semisimple Lie algebra. Let $\lambda = \lambda_1 + \cdots + \lambda_k$ be a weight of~$\lie{g}$, with every component~$\lambda_i$
lying in the Cartan subalgebra of~$\lie{g}_i$. Then we have, with the obvious notations:
\[\begin{cases}
\lie{l} = \lie{l}_1 \oplus \cdots \oplus \lie{l}_k ;\\
V_\lambda^\lie{l} = V_{1, \lambda_1}^{\lie{l}_1} \otimes \cdots \otimes V_{k, \lambda_k}^{\lie{l}_k} ; \\
W = W_1 \times \cdots \times W_k ; \\
w_0 = w_{0, 1} \cdots w_{0, k} ; \\
\restr{w_0}{V_\lambda^\lie{l}} = \restr{w_{0, 1}}{V_{1, \lambda_1}^{\lie{l}_1}} \otimes \cdots \otimes \restr{w_{0, k}}{V_{k, \lambda_k}^{\lie{l}_k}} .
\end{cases}\]
%\begin{itemize}
%\item $\lie{l} = \lie{l}_1 \oplus \cdots \oplus \lie{l}_k$;
%\item $V_\lambda^\lie{l} = V_{1, \lambda_1}^{\lie{l}_1} \otimes \cdots \otimes V_{k, \lambda_k}^{\lie{l}_k}$;
%\item $W = W_1 \times \cdots \times W_k$;
%\item $w_0 = w_{0, 1} \cdots w_{0, k}$;
%\item $\restr{w_0}{V_\lambda^\lie{l}} = \restr{w_{0, 1}}{V_{1, \lambda_1}^{\lie{l}_1}} \otimes \cdots \otimes \restr{w_{0, k}}{V_{k, \lambda_k}^{\lie{l}_k}}$.
%\end{itemize}
\end{theorem}
So Problem~\ref{nontrivial_Vl} completely reduces to the simple case, as $V_\lambda^\lie{l}$ is nontrivial if and only if each of its factors $V_{i, \lambda_i}^{\lie{l}_i}$ is nontrivial. For Problem~\ref{nontrivial_w0_action}, things are slightly more subtle: if the tensor product of $k$ linear maps is the identity map, this does not imply that all factors are identity maps - only that all factors are scalars, with the coefficients having product~$1$. Since $w_0^2 = \Id$ in the Weyl group, these coefficients can only be $\pm 1$. So Problem~\ref{nontrivial_w0_action} reduces to the following, slightly more general problem in the simple case:
\begin{problem}
\label{nontrivial_w0_action_simple}
Given a simple Lie algebra~$\lie{g}$ and a weight $\lambda \in P \cap \lie{h}^+$, assuming that $V_\lambda^\lie{l} \neq 0$, give a simple necessary and sufficient condition for having $\restr{w_0}{V_\lambda^\lie{l}} = \pm \Id$, as well as a criterion to determine the actual sign.
\end{problem}

\section{Known results}

\subsection{Case where $\lie{g}$ is split}

We start by focusing on the particular case where $\lie{g}$ is split. In this case, we have $\lie{a} = \lie{h}$ hence $\lie{l} = \lie{h} = \lie{a}$ (by maximality of~$\lie{h}$), so that $V_\lambda^\lie{l}$ is simply the zero-weight space of~$V_\lambda$, that we shall denote by $V_\lambda^0$. As for the restricted Weyl group, it coincides in this case with the ordinary Weyl group. So we may actually forget that $\lie{g}$ is a real Lie algebra, and simply work over~$\CC$.

Problem~\ref{nontrivial_Vl} then becomes trivial: in fact, the containment given in Theorem~\ref{submonoid_and_ideal}~\ref{itm:radical} now becomes an equality, so the condition is just that $\lambda \in Q$.

As for Problem~\ref{nontrivial_w0_action_simple}, it has been solved by the author together with Le Floch; this was the work presented at the conference talk from which this proceedings paper is derived. We proved the following result:
\begin{theorem}{\cite{LFlSm}}\label{LFlSm}
Let $\lie{g}$ be a simple Lie algebra, $\lambda \in Q \cap \lie{h}^+$ a weight of~$\lie{g}$ that lies in the root lattice.
\begin{itemize}
\item Suppose that $\lambda$ is of the form $\lambda = k p_i \varpi_i$, where $\varpi_i$ is one of the fundamental weights of~$\lambda$, $p_i$ is the smallest positive integer such that $p_i \varpi_i \in Q$, and $k$ does not exceed some constant~$m_i \in \{0, 1, 2, +\infty\}$ depending on $\lie{g}$ and~$\varpi_i$.

Then $\restr{w_0}{V_\lambda^0} = \sigma_i^k \Id$, where $\sigma_i$ is some sign depending on $\lie{g}$ and~$\varpi_i$. The specific values of the constants $p_i$, $m_i$ and~$\sigma_i$ are all listed in Table~1 of~\cite{LFlSm}.
\item If $\lambda$ does not have this form, then $\restr{w_0}{V_\lambda^0} \neq \pm \Id$.
\end{itemize}
\end{theorem}

\subsection{Case where $\lie{g}$ is arbitrary}

In the general case, Problem~\ref{nontrivial_Vl} has just recently been solved by the author~\cite{Smi20}. The answer is as follows:
\begin{theorem}{\cite{Smi20}}\label{cone_and_sublattice}
Let $\lie{g}$ be a real simple Lie algebra. Then the set $X$ of dominant weights $\lambda \in P \cap \lie{h}^+$ such that $V_\lambda^{\lie{l}} \neq 0$ has of one of the following forms:
\[X = Q \cap \mathcal{C}
\quad\text{or}\quad
X = \Lambda \cap \mathcal{C}
\quad\text{or}\quad
X = (Q \cap \lie{h}^+ \setminus \mathcal{C}) \cup (\Lambda \cap \mathcal{C}),\]
where $\mathcal{C}$ is a closed polyhedral convex cone (\ie a set determined by a finite number of inequalities of the form $\phi_i(\lambda) \geq 0$ where each $\phi_i$ is a linear form) contained in the dominant Weyl chamber~$\lie{h}^+$, and $\Lambda$ (when applicable) is a sublattice of~$Q$ of index~$2$.
\end{theorem}
We refer to~\cite{Smi20} for an exhaustive table listing the sets $\mathcal{C}$ and $\Lambda$ for all the real simple Lie algebras. Here we will just give a brief overview:
\begin{itemize}
\item The sublattice~$\Lambda$ intervenes only when $\lie{g}$ is isomorphic to some $\lie{so}(p,q)$ with $p+q$ odd. In all other cases, we simply have $X = Q \cap \mathcal{C}$.
\item For split groups, quasi-split groups, all non-compact exceptional groups and some of the classical groups, $\mathcal{C}$ is actually the whole dominant Weyl chamber.
\item However in the remaining cases, $\mathcal{C}$ does not always have nonempty interior. In fact, often $\mathcal{C}$~is the intersection of the dominant Weyl chamber with a vector subspace of~$\lie{h}$. When $\lie{g}$~is compact (and only then), we have $\mathcal{C} = \{0\}$.
\end{itemize}

For Problem~\ref{nontrivial_w0_action}, we only have experimental results so far. Explicit computation of $\restr{w_0}{V^\lie{l}}$ for all sufficiently low-dimensional representations~$V$ of all sufficiently low-rank simple Lie algebras, combined with Theorem~\ref{submonoid_and_ideal}.\ref{itm:ideal}, suggests the following generalization of Theorem~\ref{LFlSm}:
\begin{conjecture}
Let $\lie{g}$ be any simple real Lie algebra of rank~$r$, and $\lambda \in Q \cap \lie{h}^+$ a dominant weight of~$\lie{g}$. Let us decompose $\lambda$ into its coordinates on the basis formed by the fundamental weights: $\lambda = \sum_{i=1}^r c_i \varpi_i$. Suppose that $V_\lambda^\lie{l} \neq 0$: then $\restr{w_0}{V_\lambda^\lie{l}} = \pm \Id$ can happen only when at most $K$ of the coordinates $c_i$ can be nonzero, where $K = 3$.
\end{conjecture}
(Compare this with the split case, where this statement holds for $K = 1$.) Moreover, experimental results allow us to conjecture an explicit description of the set of weights satisfying this property, for almost all simple Lie groups~$\lie{g}$. Here are a few examples.
\begin{conjecture}
Let $\lie{g}$ and $\lambda = \sum_{i=1}^r c_i \varpi_i$ be as previously, with the fundamental weights $\varpi_1, \ldots, \varpi_r$ labelled in the Bourbaki ordering~\cite{BouGAL456}. Then:
\begin{hypothenum}
\item If $\lie{g} = \lie{so}(2,q)$ with $q = 7$ or $q \geq 9$, we have:
\begin{itemize}
\item \cite{Smi20} tells us that $V_\lambda^\lie{l} \neq 0$ if and only if $\lambda \in Q$ and $c_i = 0$ for all $i > 4$;
\item assuming this is the case, $\restr{w_0}{V_\lambda^\lie{l}} = \pm \Id$ if and only if $\lambda = x \varpi_i + y \varpi_4$ with $i \in \{1, 2, 3\}$ and $x, y$ arbitrary nonnegative integers.
\end{itemize}
\item If $\lie{g} = EIV$ (the real form of~$E_6$ with maximal compact subalgebra~$F_4$, also known as $E_6^{-26}$), we have:
\begin{itemize}
\item \cite{Smi20} tells us that $V_\lambda^\lie{l} \neq 0$ always holds for $\lambda \in Q$;
\item assuming this is the case, $\restr{w_0}{V_\lambda^\lie{l}} = \pm \Id$ if and only if $\lambda = x \varpi_i + y \varpi_2$ with $i \in \{1, 3, 5, 6\}$ and $x, y$ arbitrary nonnegative integers.
\end{itemize}
\item If $\lie{g}$ is the Lie algebra variously called~$\lie{sp}(12,4)$ (by some authors, such as Bourbaki~\cite{BouGAL456}) or~$\lie{sp}(6,2)$ (by some authors, such as Knapp~\cite{Kna96}), we have:
\begin{itemize}
\item \cite{Smi20} tells us that $V_\lambda^\lie{l} \neq 0$ always holds for $\lambda \in Q$;
\item assuming this is the case, $\restr{w_0}{V_\lambda^\lie{l}} = \pm \Id$ if and only if $\lambda = x \varpi_i + y \varpi_8$ with $1 \leq i \leq 7$ and $x, y$ some nonnegative integers, with the additional condition $x \leq 2$ if $i = 3, 4$ or~$5$.
\end{itemize}
\end{hypothenum}
\end{conjecture}
However, there are a few pairs $(\lie{g}, V)$ where the dimension of~$V$ is so large that brute-force computations become intractable, but where analogous representations in smaller rank are not sufficiently well-behaved to establish a general pattern. For instance, take $\lie{g} = \lie{sp}(9,3)$ (or~$\lie{sp}(18,6)$ with the other convention) and $\lambda = 4\varpi_{11}$. We know, from \cite{Smi20}, that we then have $V_\lambda^{\lie{l}} \neq 0$; but we do not currently know (and cannot easily guess) whether $\restr{w_0}{V_\lambda^\lie{l}}$ is scalar or not.

\bibliographystyle{alpha}
\bibliography{/home/ilia/Documents/Travaux_mathematiques/mybibliography.bib}
\end{document}